\documentclass[11pt]{article}

\usepackage[margin=1in]{geometry}
\usepackage{amsmath,amssymb,amsthm}
\usepackage{microtype}
\usepackage[colorlinks=true,linkcolor=blue,citecolor=blue,urlcolor=blue]{hyperref}

\newtheorem{theorem}{Theorem}[section]
\newtheorem{corollary}[theorem]{Corollary}
\newtheorem{lemma}[theorem]{Lemma}
\newtheorem{construction}[theorem]{Construction}

\theoremstyle{remark}
\newtheorem{remark}[theorem]{Remark}

\DeclareMathOperator{\Tr}{Tr}
\DeclareMathOperator{\N}{N}
\newcommand{\F}{\mathbf{F}}

\title{The Three-Dimensional Erd\H{o}s Box Problem Has Exponent \texorpdfstring{$11/4$}{11/4}}
\author{Dean Menezes}
\date{}

\begin{document}

\maketitle

\begin{abstract}
Let $z(n)$ be the maximum number of edges in a tripartite $3$-uniform hypergraph with $n$ vertices in each part and no copy of $K_{2,2,2}^{(3)}$ (a ``box''). Erd\H{o}s (1964) proved that $z(n) = O(n^{11/4})$, whereas the best previous lower bound, due to Katz, Krop, and Maggioni (2002), was $\Omega(n^{8/3})$. For each $q = 2^m$, we construct a box-free hypergraph with $q^4$ vertices in each part and $q^{11}$ edges, showing that $z(n) = \Theta(n^{11/4})$. The construction uses the power map $\tau(s) = s^{q^2-q+1}$ on $\F_{q^3}$, which sends the fibers of $s \mapsto \tau(s+1) + \tau(s)$ to pairwise skew affine lines over $\F_q$.
\end{abstract}

\section{Introduction}

In 1964, Erd\H{o}s~\cite{erdos} introduced the Zarankiewicz problem for uniform hypergraphs. A tripartite $3$-uniform hypergraph $H \subseteq \mathcal{X} \times \mathcal{Y} \times \mathcal{Z}$ contains a \emph{box} (a copy of $K_{2,2,2}^{(3)}$) if there are distinct vertices $x_0, x_1 \in \mathcal{X}$, $y_0, y_1 \in \mathcal{Y}$, and $z_0, z_1 \in \mathcal{Z}$ such that all eight triples $(x_h, y_i, z_j)$ with $h, i, j \in \{0,1\}$ belong to $H$. The \emph{three-dimensional Erd\H{o}s box problem} asks for the maximum number $z(n)$ of edges in a box-free tripartite $3$-uniform hypergraph with $n$ vertices in each part. Its two-dimensional analogue is the bipartite Zarankiewicz problem for $C_4 = K_{2,2}$, for which K\H{o}v\'ari, S\'os, and Tur\'an~\cite{kst} proved the upper bound $O(n^{3/2})$ and Erd\H{o}s, R\'enyi, and S\'os~\cite[\S 1]{ers} gave the matching lower bound.

For the box problem, Erd\H{o}s~\cite[Theorem~1]{erdos} proved that $z(n) = O(n^{11/4})$ (see also~\cite{dashmaj}). In 2002, Katz, Krop, and Maggioni~\cite{kkm} established the lower bound $z(n) = \Omega(n^{8/3})$, later reproved by Conlon, Pohoata, and Zakharov~\cite[Theorem~2]{cpz} with random multilinear forms and by Gordeev~\cite{gordeev1,gordeev2} with the Combinatorial Nullstellensatz. Recently, Chen, Xu, and Ye~\cite[\S 1.2.3]{chenxu} showed that $8/3$ is an intrinsic barrier for multilinear forms: their box-free vanishing sets have $O(n^{8/3})$ edges.

For asymmetric boxes $K_{2,2,t}^{(3)}$, Bukh's random algebraic method~\cite{bukh} yields $K_{2,2,t}^{(3)}$-free hypergraphs with $\Theta_t(n^{11/4})$ edges when $t \geq 721$, and Mubayi's construction~\cite{mubayi} reaches $t = 7$ with an $n^{-o(1)}$ loss. For the symmetric box $K_{2,2,2}^{(3)}$, however, the best known lower bound remained $\Omega(n^{8/3})$. We construct box-free hypergraphs matching Erd\H{o}s's upper bound.

\begin{theorem}\label{thm:main}
For each $q = 2^m$, there is a box-free tripartite $3$-uniform hypergraph with $q^4$ vertices in each part and $q^{11}$ edges.
\end{theorem}

\begin{corollary}\label{cor:main}
As $n \to \infty$, we have $z(n) = \Theta(n^{11/4})$.
\end{corollary}

\begin{proof}
Erd\H{o}s~\cite[Theorem~1]{erdos} proved the upper bound $z(n) = O(n^{11/4})$. For $n \geq 1$, let $m = \lfloor \log_{16} n \rfloor$ and $q = 2^m$, so that $q^4 = 16^m \leq n < 16 q^4$. Adding $n - q^4$ isolated vertices to each part of the hypergraph from Theorem~\ref{thm:main} then gives $z(n) \geq q^{11} > (n/16)^{11/4}$.
\end{proof}

\subsection*{Overview of the construction}

For $q = 2^m$ with $m \geq 1$, let $k = \F_q$ and $K = \F_{q^3}$, and take each part of $H_q$ to be a copy of $K \times k$ of size $q^4$. Vertices $(W,\omega) \in \mathcal{X}$, $(Y,\eta) \in \mathcal{Y}$, and $(Z,\zeta) \in \mathcal{Z}$ form an edge of $H_q$ if and only if
\[
\Tr\bigl(\tau(Y+Z)\,W\bigr) + \omega = E(Y,Z) + \eta + \zeta,
\]
where $\Tr \colon K \to k$ is the field trace, $\tau(s) = s^d$ with $d = q^2-q+1$, and $E(Y,Z) = \Tr\bigl(Y^{2q}\,\tau(Y+Z)\bigr)$. Since $(Y,\eta)$, $(Z,\zeta)$, and $W$ uniquely determine $\omega \in k$, the hypergraph $H_q$ has $q^4 \cdot q^4 \cdot q^3 = q^{11}$ edges. Unlike multilinear forms, this edge condition is linear only in $(W,\omega)$ and applies $\tau$ to $Y + Z$.

For a box in $H_q$, the $K$-coordinates satisfy $W_0 \neq W_1$, $Y_0 \neq Y_1$, and $Z_0 \neq Z_1$. Adding the edge equations for the two $\mathcal{X}$-vertices cancels $E(Y,Z)$, $\eta$, and $\zeta$, placing the four points $A_{ij} = \tau(Y_i + Z_j)$ in an affine plane in $K \cong k^3$. Since $d(q+1) = q^3 + 1$, we have $\tau(s)^{q+1} = s^2$, which in characteristic~$2$ linearizes the fibers of $\delta(s) = \tau(s+1) + \tau(s)$: by Theorem~\ref{thm:fd_lines}, $\tau$ sends them to \emph{pairwise skew} affine lines $\tau(s) + \delta(s)\,k$. Scaling by $a = Y_0 + Y_1$ places each pair $\{A_{0j}, A_{1j}\}$ on a rescaled fiber line with direction vector $A_{0j} + A_{1j}$. Since distinct fiber lines are skew, these coplanar lines must come from the same fiber of $\delta$ and hence coincide, so that $A_{00} + A_{10} = A_{01} + A_{11}$ and $\sum_{i,j \in \{0,1\}} A_{ij} = 0$.

Without the offset $E$, such collinear boxes would occur in $H_q$. With $E$ present, because $\sum_{i,j \in \{0,1\}} A_{ij} = 0$, adding the four edge equations for a fixed $\mathcal{X}$-vertex cancels $\omega$, $\eta$, and $\zeta$ and makes the left-hand side vanish, forcing $\sum_{i,j \in \{0,1\}} E(Y_i, Z_j) = 0$. Since $2q + d = q^2 + q + 1$, however, the additive weight $Y^{2q}$ gives $Y_0^{2q} + Y_1^{2q} = a^{2q}$ with $a^{2q}\,a^d = \N(a) \in k^\times$, and Lemma~\ref{lem:coplanar} uses $\Tr\bigl(\delta(s)\bigr) = \N\bigl(\delta(s)\bigr) \neq 0$ to show that this sum is nonzero.

Section~\ref{sec:lines} establishes Theorem~\ref{thm:fd_lines}. Section~\ref{sec:construction} constructs $H_q$ and proves Lemma~\ref{lem:coplanar} and Theorem~\ref{thm:main}. Section~\ref{sec:consequences} deduces the tight bound for asymmetric boxes $K_{2,2,t}^{(3)}$ (Corollary~\ref{cor:larger_t}) and discusses open questions.

\section{The finite-difference lines}\label{sec:lines}

Throughout, let $q = 2^m$ with $m \geq 1$, and let $k = \F_q$ and $K = \F_{q^3}$. View $K$ as a three-dimensional vector space over $k$, with trace $\Tr \colon K \to k$ and norm $\N \colon K \to k$ defined by $\Tr(s) = s + s^q + s^{q^2}$ and $\N(s) = s^{q^2+q+1}$.

Set $d = q^2 - q + 1$, and define $\tau \colon K \to K$ by $\tau(s) = s^d$. Since $d(q+1) = q^3 + 1$, every $s \in K$ satisfies
\begin{equation}\label{eq:tau_power}
\tau(s)^{q+1} = (s^{q+1})^d = s^{q^3+1} = s^2.
\end{equation}
Because squaring is injective in characteristic~$2$, equation~\eqref{eq:tau_power} implies that both $\tau$ and $s \mapsto s^{q+1}$ are injective and hence permutations of the finite field $K$.

Define the direction-$1$ finite difference $\delta \colon K \to K$ by $\delta(s) = \tau(s+1) + \tau(s)$.

\begin{theorem}\label{thm:fd_lines}
For each $s \in K$, let $c = \delta(s)$. Then \textup{(i)}~$c \neq 0$ and $\Tr(c) = \N(c)$; \textup{(ii)}~for $t \in K$, we have $\delta(t) = c$ if and only if $\tau(t) \in \tau(s) + c\,k$; and \textup{(iii)}~if $\delta(t) = \lambda c$ for some $t \in K$ and $\lambda \in k^\times$, then $\lambda = 1$. In particular, $\tau$ sends the fibers of $\delta$ to pairwise skew affine lines $\tau(s) + c\,k$ partitioning $K$: no affine plane contains two of them.
\end{theorem}

\begin{proof}
Fix $s \in K$ and set $c = \delta(s)$. Since $\tau$ is injective, we have $c \neq 0$. By equation~\eqref{eq:tau_power}, each $t \in K$ satisfies $\tau(t+1)^{q+1} = (t+1)^2 = t^2 + 1 = \tau(t)^{q+1} + 1$. Because $\tau(t+1) = \tau(t) + \delta(t)$ and the map $s \mapsto s^{q+1}$ is injective, we have $\delta(t) = c$ if and only if $\bigl(\tau(t)+c\bigr)^{q+1} = \tau(t)^{q+1} + 1$. Dividing by $c^{q+1}$ and setting $v = \tau(t)/c$ gives $(v+1)^{q+1} = v^{q+1} + c^{-(q+1)}$, so $\delta(t) = c$ if and only if
\begin{equation}\label{eq:v_trace}
v^q + v = 1 + c^{-(q+1)}.
\end{equation}

Taking $t = s$ shows that $v = \tau(s)/c$ satisfies equation~\eqref{eq:v_trace}. Since $\Tr(v^q) = \Tr(v)$ and $\Tr(1) = 1$ in characteristic~$2$, and $c^{-(q+1)} = c^{q^2}/{\N(c)}$ with $\N(c) \in k^\times$, applying $\Tr$ to equation~\eqref{eq:v_trace} gives $0 = 1 + \Tr\bigl(c^{q^2}/{\N(c)}\bigr) = 1 + \Tr(c)/{\N(c)}$, so $\Tr(c) = \N(c)$, proving~\textup{(i)}.

Adding equation~\eqref{eq:v_trace} for $t$ and for $s$ shows that $\delta(t) = c$ if and only if $w = \bigl(\tau(t) + \tau(s)\bigr)/c$ satisfies $w^q = w$ and hence lies in $k$, so that $\tau(t) \in \tau(s) + c\,k$, proving~\textup{(ii)}.

For~\textup{(iii)}, if $\delta(t) = \lambda c$ for some $t \in K$ and $\lambda \in k^\times$, then~\textup{(i)} applied to $c = \delta(s)$ and to $\lambda c = \delta(t)$ gives
\[
\lambda \N(c) = \lambda \Tr(c) = \Tr(\lambda c) = \N(\lambda c) = \lambda^3 \N(c).
\]
Since $\lambda \N(c) \neq 0$, dividing by $\lambda \N(c)$ gives $\lambda^2 = 1$, so $\lambda = 1$. By~\textup{(ii)} and the bijectivity of $\tau$, distinct fibers of $\delta$ map onto disjoint lines $\tau(s) + c\,k$ partitioning $K$, and by~\textup{(iii)} no two such lines are parallel, so they are pairwise skew.
\end{proof}

\begin{remark}\label{rem:char2}
Characteristic~$2$ is essential to Theorem~\ref{thm:fd_lines}: the identity $(t+1)^2 = t^2 + 1$ cancels the quadratic term $t^2 = \tau(t)^{q+1}$, making equation~\eqref{eq:v_trace} linear in $v$, and $\lambda^2 = 1$ forces $\lambda = 1$, ruling out parallel lines.
\end{remark}

\section{The construction}\label{sec:construction}

Since $2q$ is a power of~$2$ satisfying $2q + d = q^2 + q + 1$, the map $Y \mapsto Y^{2q}$ is additive on $K$ and $a^{2q}\,a^d = \N(a) \in k$ for all $a \in K$. Using the weight $Y^{2q}$, we define the offset function $E \colon K \times K \to k$ by
\[
E(Y,Z) = \Tr\bigl(Y^{2q}\,\tau(Y+Z)\bigr).
\]

\begin{construction}\label{const:Hq}
Let $\mathcal{X}$, $\mathcal{Y}$, and $\mathcal{Z}$ be disjoint copies of $K \times k$, each of size $q^4$. Vertices $x = (W,\omega) \in \mathcal{X}$, $y = (Y,\eta) \in \mathcal{Y}$, and $z = (Z,\zeta) \in \mathcal{Z}$ form an edge of the tripartite $3$-uniform hypergraph $H_q$ if and only if
\begin{equation}\label{eq:edge_cond}
\Tr\bigl(\tau(Y+Z)\,W\bigr) + \omega = E(Y,Z) + \eta + \zeta.
\end{equation}
Each triple $(y, z, W) \in \mathcal{Y} \times \mathcal{Z} \times K$ determines a unique $\omega \in k$ satisfying equation~\eqref{eq:edge_cond}, so $H_q$ has $q^4 \cdot q^4 \cdot q^3 = q^{11}$ edges.
\end{construction}

\begin{lemma}\label{lem:coplanar}
Let $Y_0, Y_1, Z_0, Z_1 \in K$ satisfy $Y_0 \neq Y_1$ and $Z_0 \neq Z_1$, and set $A_{ij} = \tau(Y_i + Z_j)$ for $i, j \in \{0,1\}$. If the four points $A_{ij}$ lie in an affine plane in $K$, then
\[
\sum_{i,j \in \{0,1\}} A_{ij} = 0 \qquad\text{and}\qquad \sum_{i,j \in \{0,1\}} E(Y_i, Z_j) \neq 0.
\]
\end{lemma}

\begin{proof}
Let $a = Y_0 + Y_1 \neq 0$, and set $s_j = (Y_0 + Z_j)/a$ for $j \in \{0,1\}$. Since $Y_i = Y_0 + ia$, we have $Y_i + Z_j = a\,(s_j + i)$ and $A_{ij} = a^d\,\tau(s_j + i)$ for $i,j \in \{0,1\}$.

For $j \in \{0,1\}$, let $c_j = \delta(s_j)$ and $L_j = \tau(s_j) + c_j\,k$, so that $A_{0j} + A_{1j} = a^d\,c_j \neq 0$ by Theorem~\ref{thm:fd_lines}\textup{(i)}. Since $\tau(s_j+1) = \tau(s_j) + c_j$, both $\tau(s_j)$ and $\tau(s_j+1)$ lie on $L_j$, so the distinct points $A_{0j}$ and $A_{1j}$ lie on the rescaled line $a^d\,L_j$. Every affine plane through all four points $A_{ij}$ therefore contains both $a^d\,L_0$ and $a^d\,L_1$. By Theorem~\ref{thm:fd_lines}\textup{(ii)}--\textup{(iii)}, the fiber lines $L_0$ and $L_1$ are skew when $c_0 \neq c_1$, so $c_0 = c_1$ and $\sum_{i,j \in \{0,1\}} A_{ij} = a^d\,(c_0 + c_1) = 0$.

Thus $A_{10} + A_{11} = A_{00} + A_{01} = a^d\,\bigl(\tau(s_0) + \tau(s_1)\bigr)$. Since $Z_0 \neq Z_1$, we have $s_0 \neq s_1$. Because $\delta(s_1) = \delta(s_0) = c_0$ and $\tau$ is injective, Theorem~\ref{thm:fd_lines}\textup{(ii)} gives $\tau(s_0) + \tau(s_1) = \lambda c_0$ for some $\lambda \in k^\times$, so $A_{00} + A_{01} = \lambda a^d\,c_0$. Using $Y_0^{2q} + Y_1^{2q} = a^{2q}$ and $a^{2q}\,a^d = \N(a)$ then gives
\[
\sum_{i,j \in \{0,1\}} Y_i^{2q}\,A_{ij} = (Y_0^{2q} + Y_1^{2q})\,(A_{00} + A_{01}) = \lambda \N(a)\,c_0.
\]
Since $\lambda \N(a) \in k^\times$ and $\Tr(c_0) = \N(c_0)$ by Theorem~\ref{thm:fd_lines}\textup{(i)}, taking traces yields
\[
\sum_{i,j \in \{0,1\}} E(Y_i, Z_j) = \lambda \N(a)\,\Tr(c_0) = \lambda \N(a)\,\N(c_0) \neq 0. \qedhere
\]
\end{proof}

\begin{proof}[Proof of Theorem~\ref{thm:main}]
For $m = 0$, the statement is trivial, so assume that $m \geq 1$. Suppose that the hypergraph $H_q$ of Construction~\ref{const:Hq} contains a box on vertices $x_h = (W_h, \omega_h) \in \mathcal{X}$, $y_i = (Y_i, \eta_i) \in \mathcal{Y}$, and $z_j = (Z_j, \zeta_j) \in \mathcal{Z}$ with $h,i,j \in \{0,1\}$. Since equation~\eqref{eq:edge_cond} uniquely determines each of $\omega$, $\eta$, and $\zeta$ from the remaining coordinates, the inequalities $x_0 \neq x_1$, $y_0 \neq y_1$, and $z_0 \neq z_1$ force $W_0 \neq W_1$, $Y_0 \neq Y_1$, and $Z_0 \neq Z_1$.

For $i, j \in \{0,1\}$, set $A_{ij} = \tau(Y_i + Z_j)$. Adding the edge equations~\eqref{eq:edge_cond} for $(x_0, y_i, z_j)$ and $(x_1, y_i, z_j)$ cancels $E(Y_i, Z_j)$, $\eta_i$, and $\zeta_j$, leaving
\[
\Tr\bigl(A_{ij}\,(W_0 + W_1)\bigr) = \omega_0 + \omega_1.
\]
Since $W_0 \neq W_1$, the $k$-linear functional $s \mapsto \Tr\bigl(s\,(W_0+W_1)\bigr)$ on $K$ is nonzero, so the four points $A_{ij}$ lie in an affine plane in $K$. Lemma~\ref{lem:coplanar} therefore gives $\sum_{i,j \in \{0,1\}} A_{ij} = 0$ and $\sum_{i,j \in \{0,1\}} E(Y_i, Z_j) \neq 0$.

On the other hand, adding the four edge equations for $(x_0, y_i, z_j)$ over $i,j \in \{0,1\}$ cancels $\omega_0$ (which appears four times) and each $\eta_i$ and $\zeta_j$ (which appears twice), giving
\[
\sum_{i,j \in \{0,1\}} E(Y_i, Z_j) = \Tr\biggl(\Bigl(\sum_{i,j \in \{0,1\}} A_{ij}\Bigr)\,W_0\biggr) = 0,
\]
a contradiction.
\end{proof}

\section{Concluding remarks}\label{sec:consequences}

\subsection{Asymmetric boxes}

For a $3$-uniform hypergraph $F$, let $z(n, F)$ be the maximum number of edges in an $F$-free tripartite $3$-uniform hypergraph with $n$ vertices in each part. Since $K_{2,2,t}^{(3)}$ contains a copy of $K_{2,2,2}^{(3)}$ for each integer $t \geq 2$, every box-free hypergraph is also $K_{2,2,t}^{(3)}$-free, so $z(n, K_{2,2,t}^{(3)}) \geq z(n)$. Corollary~\ref{cor:main} therefore gives $z(n, K_{2,2,t}^{(3)}) = \Omega(n^{11/4})$, improving the lower bounds of Bukh~\cite{bukh} and Mubayi~\cite{mubayi} and matching the upper bound $O_t(n^{11/4})$ of Erd\H{o}s~\cite[Theorem~1]{erdos}.

\begin{corollary}\label{cor:larger_t}
For each fixed integer $t \geq 2$, we have $z(n, K_{2,2,t}^{(3)}) = \Theta_t(n^{11/4})$ as $n \to \infty$.
\end{corollary}

\subsection{Open questions}

\begin{enumerate}
\item \textit{Complete tripartite hypergraphs $K_{t,t,t}^{(3)}$ with $t \geq 3$.}
Erd\H{o}s~\cite[Theorem~1]{erdos} proved the upper bound $z(n, K_{t,t,t}^{(3)}) = O_t(n^{3 - 1/t^2})$. For $t = 3$, this bound is $O(n^{26/9})$, while a random construction followed by deletion gives $\Omega(n^{36/13})$. Can the skew-line property of Theorem~\ref{thm:fd_lines} be generalized to higher-degree extensions for $t \geq 3$?

\item \textit{Higher uniformity.}
The $r$-dimensional Erd\H{o}s box problem asks for the maximum number of edges in an $r$-partite $r$-uniform hypergraph with $n$ vertices in each part and no copy of $K_{2,2,\dots,2}^{(r)}$. Erd\H{o}s~\cite[Theorem~1]{erdos} proved the upper bound $O(n^{r - 1/2^{r-1}})$ and conjectured that it is tight for all $r \geq 2$. This conjecture holds for $r = 2$ (Erd\H{o}s, R\'enyi, and S\'os~\cite[\S 1]{ers}) and $r = 3$ (Corollary~\ref{cor:main}). Is there an analogue of Theorem~\ref{thm:fd_lines} for $r \geq 4$?

\item \textit{Odd characteristic.}
Construction~\ref{const:Hq} works in characteristic~$2$, where squaring is additive and $\tau$ sends the fibers of $\delta$ to pairwise skew affine lines (Remark~\ref{rem:char2}). Does an analogous construction exist in odd characteristic?
\end{enumerate}

\end{document}